\documentclass{amsart}%
\usepackage{amsfonts}
\usepackage{amsmath}
\usepackage{amssymb}
\usepackage{graphicx}%
\newtheorem{theorem}{Theorem}
\theoremstyle{plain}

\newtheorem{definition}{Definition}

\newtheorem{lemma}{Lemma}

\newtheorem{proposition}{Proposition}
\newtheorem{remark}{Remark}

\numberwithin{equation}{section}

\usepackage{times}
\usepackage{color}
\usepackage{pstricks-add}
\usepackage[font={scriptsize,it},textfont=it]{caption}

\begin{document}
\title{\textbf{Densities non-realizable as the Jacobian of a 2-dimensional 
bi-Lipschitz map are generic}}
\author{Rodolfo Viera}
\maketitle

\vspace{-1cm}

\begin{abstract}In this work, positive functions defined on the plane are considered from a generic viewpoint, 
both in the continuous and the bounded setting. By pursuing on constructions of Burago-Kleiner and 
McMullen, we show that, generically, such a function cannot be written as the Jacobian of a bi-Lipschitz 
homeomorphism.
\end{abstract}

\section{Introduction}

A direct consequence of the Fundamental Theorem of Calculus is that every positive continuous function 
$\rho$ defined on a compact interval $[a,b]$ can be written as the derivative of a diffeomorphism, namely 
$$f(x) := \int_a^x \rho (s) ds.$$
The very same formula shows that every positive $L^{\infty}$ function that is bouded away from zero 
can be written as the a.e. derivative of a bi-Lipschitz homeomorphism. 

As was shown by Burago-Kleiner \cite{BaK} and McMullen \cite{McM} in very important works 
(see also \cite{VK}), this is no longer true in the 2-dimensional framework: there exist positive 
$L^{\infty}$ (even continuous) functions that cannot be written as the Jacobian of a bi-Lipschitz 
homeomorphism of the plane. This pure analytical result is obtained as the fundamental step of another  
result of a discrete nature: there exist coarsely dense, uniformly discrete sets in the plane (that can be taken 
as subsets of $\mathbb{Z}^2$) that are not bi-Lipschitz equivalent to the standard lattice $\mathbb{Z}^2$. 
Although recently a shortcut (and extension) to this last result has been produced in \cite{CN}, the 
analytical one is interesting by itself, and deserves more attention. Based on Burago-Kleiner's version 
of this result (which is slightly more general than McMullen's), in this work we show that not only bad 
densities exist, but are also generic.

\vspace{0.4cm} 

\noindent{\bf Main Theorem.} A generic positive continuous function $\phi: \mathbb{R}^2 \to \mathbb{R}$ 
cannot be written as the Jacobian of a bi-Lipschitz homeomorphism. The same holds for a generic positive 
$L^{\infty}$ function.

\vspace{0.4cm}

We point out that similar results hold (with the same proof) in dimension greater than 2. The 
restriction to the 2-dimensional case below just allows simplifying notation and computations.

%%%%%%%%%%%%%%%%%%%%%%%%%%%

\section{The Burago-Kleiner construction}

In this section, we review the Burago-Kleiner construction of a density that is nonrealizable as the Jacobian 
of a bi-Lipschitz map. We begin with some definitions and classical results.

\begin{definition} A map $f:(X_1,d_1)\longrightarrow (X_2,d_2)$ between two metric 
spaces is said to be {\bf $L$-bi-Lipschitz} if for every pair of points $x,y\in X_1$, we have
\begin{center}
$\frac{1}{L}\cdot d_1(x,y)\leq d_2(f(x),f(y))\leq L\cdot d_1(x,y)$.
\end{center}
We call the map $f$ {\bf bi-Lipschitz} if there exists some $L\geq 1$ such that $f$ is $L$-bi-Lipschitz.
\end{definition}

\vspace{0.1cm}

\begin{definition}
We say that two points $x,y\in\mathbb{R}^2$ are {\bf $A$-stretched} under a map $f:\mathbb{R}^2\longrightarrow\mathbb{R}^2$ if $||f(x)-f(y)||\geq A||x-y||$.\\
\end{definition}

We remind a classical result, a proof of which may be found in an Appendix of \cite{Gr}.\\

\noindent\textbf{Theorem (Rademacher).} Let $U\subset\mathbb{R}^n$ be open and $f:U\longrightarrow\mathbb{R}^m$ 
a Lipschitz map. Then $f$ is almost everywhere differentiable in $U$, and its partial derivatives belong to $L^{\infty}$.\\

We denote by $I^2:=[0,1]\times[0,1]$ the standard unit square in $\mathbb{R}^2$. The key tool for the existence of a {\em bad density}, that is, 
a density that is not realizable as the Jacobian of a bi-Lipschitz map, is given in the next proposition from \cite{BaK}.\\

\begin{proposition}For any given $L>1$ and $c>0$, there exists a continuous function $\rho: I^2\longrightarrow [1,1+c]$ such that there is no $L$-bi-Lipschitz homeomorphism $\phi:I^2\longrightarrow\mathbb{R}^2$ satisfying

\begin{center}
$Jac(\phi)=\rho$ a.e.
\end{center}
\end{proposition}

\vspace{0.2cm}

We next sketch the construction of the density above. First of all, Burago and Kleiner observe that it is sufficient to construct a measurable density 
$\rho_*: I^2 \to [1,1+c]$ having the property above. Indeed, if $(\rho_k)_{k\in\mathbb{N}}$ is a sequence of smoothings of the measurable function $\rho_*$ 
that converges to $\rho_*$ in $L^1$, then by the Arzel\`a-Ascoli theorem, any sequence of $L$-bi-Lipschitz maps $\phi_k:I^2\longrightarrow\mathbb{R}^2$ 
such that $Jac(\phi_k)=\rho_k$ a.e will converge, up to a subsequence, to a bi-Lipschitz map $\phi:I^2\longrightarrow\mathbb{R}^2$ satisfying 
$Jac(\phi)=\rho_*$ a.e. See the proof of Proposition \ref{prop:smoothings} for a variation of this argument.

%Note that given $\varepsilon>0$, this process provide us a smooth function $\rho_{k_0}:I^2\longrightarrow\mathbb{R}$ for some $k_0\in\mathbb{N}$ with the same 
%properties as the Proposition 1 and such that $||\rho_{k_0}-\rho||_{L^1}<\varepsilon$. To obtain a continuous function $\varphi_{k_0}:I^2\longrightarrow\mathbb{R}$ 
%such that cannot be realized as the Jacobian of a $L$-bi-Lipschitz map satisfying $||\varphi_{k_0}-\rho||_{\infty}<\varepsilon$, is sufficient to consider  $\varphi_{k_0}=\mathcal{H}_t*%\rho_{k_0}$, where $\mathcal{H}_t$ is a \textit{Heat Kernel} for some suitable $t>0$. \\

The key tool for the construction of the measurable bad density is the \textbf{checkerboard function} 
$\rho_{N,c}$ (for $N=N(L,c)$ sufficiently large) defined below. 

\vspace{0.4cm}

\begin{definition}
For $N\in\mathbb{N}$, consider the rectangle $R_N:=[0,1]\times \left[0,\frac{1}{N}\right]$ and the squares $S_i:=\left[\frac{i-1}{N},\frac{i}{N}\right]\times \left[0,\frac{1}{N}\right]$, 
with $i=1,\ldots, N$. Given $c>0$, we define the function $\rho_{N,c}: R_N\longrightarrow [1,1+c]$ by letting 
\begin{center}
$\rho_{N,c}(x) = \left\{
\begin{array}{c l}
 1 \qquad \mbox{   if }  x\in S_i\mbox{ with } i \mbox{ even},\\
 \\
 1+c \quad \mbox{   if } x\in S_i\mbox{ with } i \mbox{ odd}.\\
\end{array}
\right.$
\end{center}

%\vspace{0.4cm}

%This function $\rho_N$ is called a \textbf{checkerboard} function.
\end{definition}

\begin{figure}[h]
\psset{unit=0.7cm,algebraic=true,dimen=middle,dotstyle=o,dotsize=3pt 0,linewidth=0.8pt,arrowsize=3pt 2,arrowinset=0.2}
\begin{pspicture}(-1.,-2.)(10.,3.)

\pspolygon(0.,0.)(0.,1.)(4.,1.)(4.,0.)
\psline(0.,0.)(0.,1.)
\psline(0.,1.)(4.,1.)
\psline(4.,1.)(4.,0.)
\psline(4.,0.)(0.,0.)
\psline(1.,0.)(1.,1.)
\psline(2.,0.)(2.,1.)
\psline(3.,0.)(3.,1.)
\psline(6.,0.44)(9.,0.44)
\pscurve{->}(0.6,0.26)(1.3,-0.66)(2.68,-1.22)(4.82,-1.18)(5.8,-0.26)
\pscurve{->}(1.6,0.8)(2.4,1.6)(3.86,2.26)(6.02,2.38)(7.76,1.86)(8.68,1.2)
\pscurve{->}(2.52,0.12)(3.42,-0.48)(4.7,-0.46)(5.46,0.14)
\pscurve{->}(3.56,0.92)(4.25,1.45)(5.45,1.55)(6.66,1.38)(7.78,1.08)(8.58,0.75)
\uput[l](6,0.44){1}
\uput[r](9,0.44){1+c}
\begin{scriptsize}
\psdots[dotstyle=*](6,0.44)
\psdots[dotstyle=*](9,0.44)
\end{scriptsize}
\end{pspicture}
\caption{The checkerboard function $\rho_{N,c}: R_N\longrightarrow [1,1+c]$.}
\end{figure}
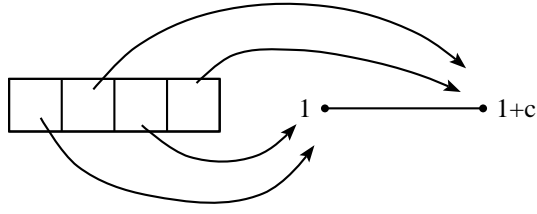

\vspace{0.4cm}

Using a quite long argument, Burago and Kleiner prove that for a large enough $N$, there exists 
$\epsilon=\epsilon(L,c)>0$ such that for every $L$-bi-Lipschitz map $\phi:I^2\longrightarrow\mathbb{R}^2$ 
whose Jacobian $Jac(\phi)$ is equal to $\rho_{N,c}$ except for a set of measure less than $\epsilon$, 
there exist two points $x,y\in R_N$ that are $(1+\kappa)||\phi(0,0)-\phi(1,0)||$-stretched under $\phi$ 
for a certain $\kappa=\kappa(L,c)>0$. Using this, as a next step they modify the function on a 
rectangular neighborhood $U$ of $\overline{xy}$ by including a rescaled version of another  
checkerboard function. By the same argument, there are two points $x_2,y_2\in U$ 
that are $(1+\kappa)^2||\phi(1,0)-\phi(0,0)||$-stretched by an homeomorphism 
realizing this new function as the Jacobian. Repeating this argument at smaller and 
smaller scales, we eventually obtain a function $\rho: R_N\longrightarrow [1,1+c]$ 
which cannot be the Jacobian of an $L$-bi-Lipschitz map.\\

The next lemma summarizes a slightly stronger version of the output of the construction above, and it also appears in \cite{BaK}. 
It guarantees the non-existence of an $L$-bi-Lipschitz map with a certain prescribed Jacobian up to a small error by providing  a 
finite collection of pairs of points that are stretched by a uniform amount under these maps.\\

\begin{lemma}Given $L>1$ and $c>0$, for each positive integer $i$ there is a continuous function $\rho_i:I^2\longrightarrow [1,1+c]$, a 
finite collection $\mathcal{S}_i$ of non-intersecting segments $\overline{l_kr_k}\subset I^2$, and $\delta_i>0$, with the following property: 
For every $L$-bi-Lipschitz map $f:I^2\longrightarrow\mathbb{R}^2$ whose Jacobian differs from $\rho_i$ on a set of area less than $\delta_i$, 
the endpoints of at least one of the segments from $\mathcal{S}_i$ are $\frac{(1+\kappa)^{i}}{L}$-stretched by $f$.
\label{lema-epsilon}
\end{lemma}

%Note that the Lemma $1$ allow us obtain the continuity of the non-realizable density, since there is $\epsilon>0$ such that we can change the values of $\rho_{L,c}$ arbitrarily on a set %of measure less than $\epsilon$.

%%%%%%%%%%%%%%%%%%%%%%%%%%%%%%%%%%%%%%%%%%%%%%%%%%%%%%%%%%%%%%%%%%%%%%%%%%%%%%%%%%%%%%%%%%

\section{The continuous case}
\label{first section}

Recall that a subset $G$ of a metric space $M$ is said to be {\bf thick} if it contains a set of the 
form $\bigcap_{n\in\mathbb{N}}G_n$, where each $G_n$ is an open dense subset of $M$.\\

\begin{definition}If all points of a thick subset have some property, then this is said to be a {\bf generic} property of $M$.\\
\end{definition}

We now consider the space $C_+(I^2,\mathbb{R})$ of all positive continuous functions $\rho:I^2\longrightarrow \mathbb{R}$ 
with the norm $||f||_0:=\sup\{|f(x)|\mbox{ : }x\in I^2\}$ . In this section, we will use the Burago-Kleiner construction to prove 
the Main Theorem in the continuous setting.\\ 

\begin{theorem}
Let $\mathcal{C}$ be the set of all functions $\rho\in C_+(I^2,\mathbb{R})$ such that there is no bi-Lipschitz 
map $\phi:I^2\longrightarrow\mathbb{R}^2$ satisfying $\rho=Jac(\phi)$ a.e. Then $\mathcal{C}$ is a thick 
subset of $C^+(I^2,\mathbb{R})$.
% there exists a family $\{\mathcal{C}_{n}\}_{n\in\mathbb{N}}$ of open dense subsets of $C^+(I^2,\mathbb{R})$ such that 
%\begin{equation*}
%\begin{split}
%\bigcap_{n\in\mathbb{N},n>1}\mathcal{C}_{n}\subset \mathcal{C}\\
%\end{split}
%\end{equation*}
\end{theorem}

\vspace{0.21cm}

This Theorem will follow from the next

\vspace{0.21cm}

\begin{proposition}\label{prop:smoothings}
Given $L>1$, consider the set $\mathcal{C}_{L}$ of all functions $\rho\in C_+(I^2,\mathbb{R})$ 
such that there is no $L$-bi-Lipschitz map $\phi:I^2\longrightarrow\mathbb{R}^2$ satisfying 
$Jac(\phi)=\rho$ a.e. Then $\mathcal{C}_{L}$ is an open dense subset of $C_+(I^2,\mathbb{R})$.
\end{proposition}

\vspace{0.2cm}

Theorem 1 is a direct consequence of this, since 
$$ \bigcap_{n\in\mathbb{N}}\mathcal{C}_{n} = \mathcal{C}.$$

To prove Proposition 2, we first establish that $\mathcal{C}_{L}$ is an open subset of 
$C_+(I^2,\mathbb{R})$ for each $L > 0$. Let $X:=C_+(I^2,\mathbb{R})\setminus\mathcal{C}_{L}$, and consider a sequence 
$(\rho_k)_{k\in\mathbb{N}}\subset X$ such that $\rho_k\xrightarrow[k\rightarrow\infty]{}\rho$. We need to show that $\rho \notin \mathcal{C}_L$. 
To do this, let $\phi_k: I^2 \longrightarrow\mathbb{R}^2$ be a sequence of $L$-bi-Lipschitz maps such that $Jac(\phi_k)=\rho_k$ a.e. By Arzel\`{a}-Ascoli 
theorem, $(\phi_k)_{k\in\mathbb{N}}$ has a subsequence that converges to an $L$-bi-Lipschitz map $\phi:I^2\longrightarrow\mathbb{R}^2$. To simplify 
notation, we assume without loss of generality that $\phi_k\xrightarrow[k\rightarrow\infty]{}\phi$. Our goal is to show that $Jac (\phi) = \rho$, 
and hence $\rho \notin \mathcal{C}_L$. 

Let $U\subset I^2$ be a closed ball and denote by $\lambda$ the Lebesgue measure in $\mathbb{R}^2$. Then, as $\rho_k \to \rho$,
\begin{equation*}
\begin{split}
Area(\phi_k(U))&=\int_U{Jac(\phi_k)d\lambda}\\
               &=\int_U{\rho_kd\lambda}\\
               &\xrightarrow[k\rightarrow\infty]{}\int_U{\rho d\lambda}.\\
\end{split}
\end{equation*}

\psset{unit=1cm,algebraic=true,dimen=middle,dotstyle=o,dotsize=3pt 0,linewidth=0.8pt,arrowsize=3pt 2,arrowinset=0.25}
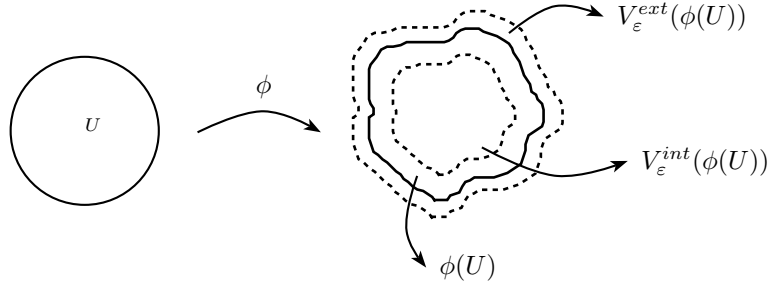
\begin{figure}[h]
\begin{pspicture}(-1.,-1.5)(10.,3.)
\pscircle(1.,1.){0.9848857801796106}
\psline[linewidth=1pt](5.64,0.08)(5.72,0.08)(5.82,0.08)(5.86,0.14)(5.92,0.16)(5.98,0.18)(6.02,0.26)(6.06,0.32)(6.12,0.34)(6.18,0.38)(6.26,0.38)(6.36,0.38)(6.44,0.38)(6.5,0.4)(6.56,0.42)(6.62,0.44)(6.68,0.48)(6.74,0.52)(6.8,0.56)(6.82,0.62)(6.84,0.68)(6.86,0.74)(6.86,0.82)(6.88,0.88)(6.9,0.94)(6.96,0.98)(7.02,1.02)(7.08,1.06)(7.1,1.12)(7.1,1.2)(7.1,1.28)(7.08,1.36)(7.02,1.42)(7.,1.48)(6.94,1.52)(6.92,1.62)(6.88,1.68)(6.84,1.74)(6.82,1.8)(6.8,1.86)(6.78,1.92)(6.74,1.98)(6.7,2.04)(6.68,2.1)(6.64,2.16)(6.58,2.2)(6.48,2.26)(6.42,2.28)(6.36,2.3)(6.3,2.32)(6.24,2.36)(6.14,2.36)(6.06,2.36)(6.,2.34)(5.94,2.3)(5.9,2.24)(5.84,2.2)(5.78,2.18)(5.7,2.18)(5.62,2.18)(5.52,2.18)(5.44,2.18)(5.36,2.18)(5.28,2.18)(5.2,2.14)(5.1,2.12)(5.02,2.12)(4.94,2.08)(4.88,1.98)(4.8,1.88)(4.78,1.82)(4.78,1.74)(4.78,1.66)(4.78,1.58)(4.78,1.5)(4.78,1.42)(4.8,1.36)(4.84,1.3)(4.84,1.22)(4.78,1.18)(4.78,1.08)(4.78,1.)(4.78,0.9)(4.8,0.84)(4.82,0.78)(4.84,0.72)(4.9,0.68)(4.96,0.62)(5.02,0.58)(5.08,0.52)(5.14,0.46)(5.2,0.42)(5.24,0.36)(5.3,0.32)(5.36,0.26)(5.42,0.22)(5.5,0.2)(5.56,0.16)(5.62,0.12)(5.66,0.06)
\psline[linewidth=1pt,linestyle=dashed,dash=2pt 2pt](5.7,0.416)(5.756,0.416)(5.826,0.416)(5.854,0.458)(5.896,0.472)(5.938,0.486)(5.966,0.542)(5.994,0.584)(6.036,0.598)(6.078,0.626)(6.134,0.626)(6.204,0.626)(6.26,0.626)(6.302,0.64)(6.344,0.654)(6.386,0.668)(6.428,0.696)(6.47,0.724)(6.512,0.752)(6.526,0.794)(6.54,0.836)(6.554,0.878)(6.554,0.934)(6.568,0.976)(6.582,1.018)(6.624,1.046)(6.666,1.074)(6.708,1.102)(6.722,1.144)(6.722,1.2)(6.722,1.256)(6.708,1.312)(6.666,1.354)(6.652,1.396)(6.61,1.424)(6.596,1.494)(6.568,1.536)(6.54,1.578)(6.526,1.62)(6.512,1.662)(6.498,1.704)(6.47,1.746)(6.442,1.788)(6.428,1.83)(6.4,1.872)(6.358,1.9)(6.288,1.942)(6.246,1.956)(6.204,1.97)(6.162,1.984)(6.12,2.012)(6.05,2.012)(5.994,2.012)(5.952,1.998)(5.91,1.97)(5.882,1.928)(5.84,1.9)(5.798,1.886)(5.742,1.886)(5.686,1.886)(5.616,1.886)(5.56,1.886)(5.504,1.886)(5.448,1.886)(5.392,1.858)(5.322,1.844)(5.266,1.844)(5.21,1.816)(5.168,1.746)(5.112,1.676)(5.098,1.634)(5.098,1.578)(5.098,1.522)(5.098,1.466)(5.098,1.41)(5.098,1.354)(5.112,1.312)(5.14,1.27)(5.14,1.214)(5.098,1.186)(5.098,1.116)(5.098,1.06)(5.098,0.99)(5.112,0.948)(5.126,0.906)(5.14,0.864)(5.182,0.836)(5.224,0.794)(5.266,0.766)(5.308,0.724)(5.35,0.682)(5.392,0.654)(5.42,0.612)(5.462,0.584)(5.504,0.542)(5.546,0.514)(5.602,0.5)(5.644,0.472)(5.686,0.444)(5.714,0.402)
\psline[linewidth=1pt,linestyle=dashed,dash=2pt 2pt](5.6,-0.144)(5.696,-0.144)(5.816,-0.144)(5.864,-0.072)(5.936,-0.048)(6.008,-0.024)(6.056,0.072)(6.104,0.144)(6.176,0.168)(6.248,0.216)(6.344,0.216)(6.464,0.216)(6.56,0.216)(6.632,0.24)(6.704,0.264)(6.776,0.288)(6.848,0.336)(6.92,0.384)(6.992,0.432)(7.016,0.504)(7.04,0.576)(7.064,0.648)(7.064,0.744)(7.088,0.816)(7.112,0.888)(7.184,0.936)(7.256,0.984)(7.328,1.032)(7.352,1.104)(7.352,1.2)(7.352,1.296)(7.328,1.392)(7.256,1.464)(7.232,1.536)(7.16,1.584)(7.136,1.704)(7.088,1.776)(7.04,1.848)(7.016,1.92)(6.992,1.992)(6.968,2.064)(6.92,2.136)(6.872,2.208)(6.848,2.28)(6.8,2.352)(6.728,2.4)(6.608,2.472)(6.536,2.496)(6.464,2.52)(6.392,2.544)(6.32,2.592)(6.2,2.592)(6.104,2.592)(6.032,2.568)(5.96,2.52)(5.912,2.448)(5.84,2.4)(5.768,2.376)(5.672,2.376)(5.576,2.376)(5.456,2.376)(5.36,2.376)(5.264,2.376)(5.168,2.376)(5.072,2.328)(4.952,2.304)(4.856,2.304)(4.76,2.256)(4.688,2.136)(4.592,2.016)(4.568,1.944)(4.568,1.848)(4.568,1.752)(4.568,1.656)(4.568,1.56)(4.568,1.464)(4.592,1.392)(4.64,1.32)(4.64,1.224)(4.568,1.176)(4.568,1.056)(4.568,0.96)(4.568,0.84)(4.592,0.768)(4.616,0.696)(4.64,0.624)(4.712,0.576)(4.784,0.504)(4.856,0.456)(4.928,0.384)(5.,0.312)(5.072,0.264)(5.12,0.192)(5.192,0.144)(5.264,0.072)(5.336,0.024)(5.432,0.)(5.504,-0.048)(5.576,-0.096)(5.624,-0.168)
\pscurve{->}(2.5,0.98)(3.38,1.26)(4.14,0.98)
\uput[u](3.38,1.26){$\phi$}
\pscurve{->}(6.64,2.3)(7.26,2.66)(7.9,2.52)
\uput[r](7.9,2.52){$V_{\varepsilon}^{ext}(\phi(U))$}
\pscurve{->}(6.28,0.78)(7.28,0.4)(8.22,0.6)
\uput[r](8.22,0.6){$V_{\varepsilon}^{int}(\phi(U))$}
\pscurve{->}(5.42,0.4)(5.28,-0.34)(5.54,-0.82)
\uput[r](5.54,-0.82){$\phi(U)$}
\begin{scriptsize}
\rput[bl](1,1){{$U$}}
\end{scriptsize}
\end{pspicture}
\caption{The sets $\phi(U)$, $V_{\varepsilon}^{ext}(\phi(U))$ and $V_{\varepsilon}^{int}(\phi(U))$.}
\end{figure}

Below we show that, also, $Area(\phi_k(U)) \to Area (\phi (U))$, and hence $Jac (\phi) = \rho$, as announced. 
To do this, given $\varepsilon > 0$, consider the sets 
$$V_{\varepsilon}^{ext}(\phi(U)):=\{x\in\mathbb{R}^2\mbox{: }d(x,\phi(U))<\varepsilon\}$$ 
and 
$$V_{\varepsilon}^{int}(\phi(U)):=\{x\in\phi(U)\mbox{: }d(x,\partial\phi(U))>\varepsilon\}.$$
Then the desired convergence 
$$Area(\phi_k(U)) \to Area (\phi (U))$$
obviously follows from the next 

\begin{lemma} 
Given $\varepsilon>0$, there exists a positive integer $k_0$ such that if $k\geq k_0$, then 
\begin{itemize}
\item[i) ] $\phi_k(U)\subset V_{\varepsilon}^{ext}(\phi(U))$,
\item[ii) ]$V_{\varepsilon}^{int}(\phi(U))\subset\phi_k(U)$.\\
\end{itemize}
\end{lemma}

\begin{proof}  Assume i) does not hold. Then for each $k\in\mathbb{N}$ there exists $y_k\in\phi_k(U) \setminus V_{\varepsilon}^{ext}(\phi(U))$. 
Let $x_k \in U$ be such that $y_k=\phi_k(x_k)$. Since $U$ is closed, there exists a sequence $(x_{k_l})_{l\in\mathbb{N}}$ that converges to a certain 
$x\in U$. By the equicontinuity of the sequence $(\phi_k)_{k\in\mathbb{N}}$ and the convergences $x_{k_l}\xrightarrow[l\rightarrow\infty]{}x$ and 
$\phi_k \to \phi$, there exists $l_0\in\mathbb{N}$ such that if $l\geq l_0$, then $|\phi_{k_l}(x_{k_l})-\phi_{k_l}(x)|<\frac{\varepsilon}{2}$ and 
$|\phi_{k_l}(x)-\phi(x)|<\frac{\varepsilon}{2}$. Thus, for $l\geq l_0$,
$$|\phi_{k_l}(x_{k_l})-\phi(x)|
\leq 
|\phi_{k_l}(x_{k_l})-\phi_{k_l}(x)|+|\phi_{k_l}(x)-\phi(x)|
 <
\dfrac{\varepsilon}{2}+\dfrac{\varepsilon}{2}
=
\varepsilon.$$
Hence, $y_{k_l}=\phi_{k_l}(x_{k_l})\xrightarrow[l\rightarrow\infty]{}\phi(x)\in\phi(U)\subset V_{\varepsilon}^{ext}(\phi(U))$. However, this is impossible, 
since $(y_k)$ is a sequence in the closed subset $\mathbb{R}^2\setminus V_{\varepsilon}^{ext}(\phi(U))$, and then all its accumulation points lie 
in $\mathbb{R}^2\setminus V_{\varepsilon}^{ext}(\phi(U))$.

To prove ii), let $x,y \in \partial U$ be such that the arc between $x$ and $y$ has length less than $\frac{\varepsilon}{2L}$. Let 
$k_0\in\mathbb{N}$ be such that $|\phi_{k}(x)-\phi(x)|<\varepsilon/4$ and $|\phi_{k}(y)-\phi(y)|<\varepsilon/4$ for all $k \geq k_0$. Since $\phi_{k}$ 
is $L$-bi-Lipschitz, the curve in the boundary of $\phi_{k}(U)$ joining $\phi_{k}(x)$ and $\phi_{k}(y)$ has a length $<\varepsilon/2$. Therefore, 
$\partial(\phi_{k}(U))\cap V_{\varepsilon}^{int}(\phi(U))=\emptyset$, and thus $V_{\varepsilon}^{int}(\phi(U))\subset\phi_k(U)$, for $k \geq k_0$.
\end{proof}

\begin{remark} \label{rem-infty}
Define $\mathcal{C}_{L,\infty}$ as being the set of all positive functions $\rho\in L^{\infty}_+(I^2,\mathbb{R})$ such that there 
is no $L$-bi-Lipschitz map $f \! : I^2\longrightarrow\mathbb{R}^2$ satisfying $Jac(f)=\rho$ a.e. Then the same argument above 
shows that $\mathcal{C}_{L,\infty}$ is an open subset of $L^{\infty}_+ (I^2,\mathbb{R})$. 
\end{remark}

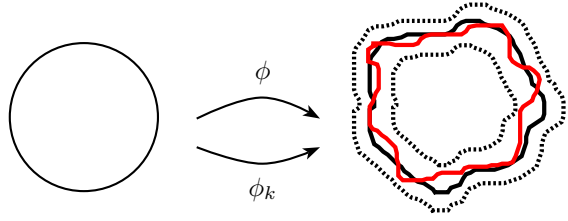
\begin{figure}[h]
\begin{pspicture*}(-1.,-1.)(8.,3.)
\pscircle(1.,1.){0.9848857801796106}
\psline[linewidth=1.6pt](5.66,0.)(5.74,0.)(5.84,0.)(5.88,0.06)(5.94,0.08)(6.,0.1)(6.04,0.18)(6.08,0.24)(6.14,0.26)(6.2,0.3)(6.28,0.3)(6.38,0.3)(6.46,0.3)(6.52,0.32)(6.58,0.34)(6.64,0.36)(6.7,0.4)(6.76,0.44)(6.82,0.48)(6.84,0.54)(6.86,0.6)(6.88,0.66)(6.88,0.74)(6.9,0.8)(6.92,0.86)(6.98,0.9)(7.04,0.94)(7.1,0.98)(7.12,1.04)(7.12,1.12)(7.12,1.2)(7.1,1.28)(7.04,1.34)(7.02,1.4)(6.96,1.44)(6.94,1.54)(6.9,1.6)(6.86,1.66)(6.84,1.72)(6.82,1.78)(6.8,1.84)(6.76,1.9)(6.72,1.96)(6.7,2.02)(6.66,2.08)(6.6,2.12)(6.5,2.18)(6.44,2.2)(6.38,2.22)(6.32,2.24)(6.26,2.28)(6.16,2.28)(6.08,2.28)(6.02,2.26)(5.96,2.22)(5.92,2.16)(5.86,2.12)(5.8,2.1)(5.72,2.1)(5.64,2.1)(5.54,2.1)(5.46,2.1)(5.38,2.1)(5.3,2.1)(5.22,2.06)(5.12,2.04)(5.04,2.04)(4.96,2.)(4.9,1.9)(4.82,1.8)(4.8,1.74)(4.8,1.66)(4.8,1.58)(4.8,1.5)(4.8,1.42)(4.8,1.34)(4.82,1.28)(4.86,1.22)(4.86,1.14)(4.8,1.1)(4.8,1.)(4.8,0.92)(4.8,0.82)(4.82,0.76)(4.84,0.7)(4.86,0.64)(4.92,0.6)(4.98,0.54)(5.04,0.5)(5.1,0.44)(5.16,0.38)(5.22,0.34)(5.26,0.28)(5.32,0.24)(5.38,0.18)(5.44,0.14)(5.52,0.12)(5.58,0.08)(5.64,0.04)(5.68,-0.02)
\psline[linewidth=1.6pt,linestyle=dashed,dash=1pt 1pt](5.714,0.36)(5.77,0.36)(5.84,0.36)(5.868,0.402)(5.91,0.416)(5.952,0.43)(5.98,0.486)(6.008,0.528)(6.05,0.542)(6.092,0.57)(6.148,0.57)(6.218,0.57)(6.274,0.57)(6.316,0.584)(6.358,0.598)(6.4,0.612)(6.442,0.64)(6.484,0.668)(6.526,0.696)(6.54,0.738)(6.554,0.78)(6.568,0.822)(6.568,0.878)(6.582,0.92)(6.596,0.962)(6.638,0.99)(6.68,1.018)(6.722,1.046)(6.736,1.088)(6.736,1.144)(6.736,1.2)(6.722,1.256)(6.68,1.298)(6.666,1.34)(6.624,1.368)(6.61,1.438)(6.582,1.48)(6.554,1.522)(6.54,1.564)(6.526,1.606)(6.512,1.648)(6.484,1.69)(6.456,1.732)(6.442,1.774)(6.414,1.816)(6.372,1.844)(6.302,1.886)(6.26,1.9)(6.218,1.914)(6.176,1.928)(6.134,1.956)(6.064,1.956)(6.008,1.956)(5.966,1.942)(5.924,1.914)(5.896,1.872)(5.854,1.844)(5.812,1.83)(5.756,1.83)(5.7,1.83)(5.63,1.83)(5.574,1.83)(5.518,1.83)(5.462,1.83)(5.406,1.802)(5.336,1.788)(5.28,1.788)(5.224,1.76)(5.182,1.69)(5.126,1.62)(5.112,1.578)(5.112,1.522)(5.112,1.466)(5.112,1.41)(5.112,1.354)(5.112,1.298)(5.126,1.256)(5.154,1.214)(5.154,1.158)(5.112,1.13)(5.112,1.06)(5.112,1.004)(5.112,0.934)(5.126,0.892)(5.14,0.85)(5.154,0.808)(5.196,0.78)(5.238,0.738)(5.28,0.71)(5.322,0.668)(5.364,0.626)(5.406,0.598)(5.434,0.556)(5.476,0.528)(5.518,0.486)(5.56,0.458)(5.616,0.444)(5.658,0.416)(5.7,0.388)(5.728,0.346)
\psline[linewidth=1.6pt,linestyle=dashed,dash=1pt 1pt](5.624,-0.24)(5.72,-0.24)(5.84,-0.24)(5.888,-0.168)(5.96,-0.144)(6.032,-0.12)(6.08,-0.024)(6.128,0.048)(6.2,0.072)(6.272,0.12)(6.368,0.12)(6.488,0.12)(6.584,0.12)(6.656,0.144)(6.728,0.168)(6.8,0.192)(6.872,0.24)(6.944,0.288)(7.016,0.336)(7.04,0.408)(7.064,0.48)(7.088,0.552)(7.088,0.648)(7.112,0.72)(7.136,0.792)(7.208,0.84)(7.28,0.888)(7.352,0.936)(7.376,1.008)(7.376,1.104)(7.376,1.2)(7.352,1.296)(7.28,1.368)(7.256,1.44)(7.184,1.488)(7.16,1.608)(7.112,1.68)(7.064,1.752)(7.04,1.824)(7.016,1.896)(6.992,1.968)(6.944,2.04)(6.896,2.112)(6.872,2.184)(6.824,2.256)(6.752,2.304)(6.632,2.376)(6.56,2.4)(6.488,2.424)(6.416,2.448)(6.344,2.496)(6.224,2.496)(6.128,2.496)(6.056,2.472)(5.984,2.424)(5.936,2.352)(5.864,2.304)(5.792,2.28)(5.696,2.28)(5.6,2.28)(5.48,2.28)(5.384,2.28)(5.288,2.28)(5.192,2.28)(5.096,2.232)(4.976,2.208)(4.88,2.208)(4.784,2.16)(4.712,2.04)(4.616,1.92)(4.592,1.848)(4.592,1.752)(4.592,1.656)(4.592,1.56)(4.592,1.464)(4.592,1.368)(4.616,1.296)(4.664,1.224)(4.664,1.128)(4.592,1.08)(4.592,0.96)(4.592,0.864)(4.592,0.744)(4.616,0.672)(4.64,0.6)(4.664,0.528)(4.736,0.48)(4.808,0.408)(4.88,0.36)(4.952,0.288)(5.024,0.216)(5.096,0.168)(5.144,0.096)(5.216,0.048)(5.288,-0.024)(5.36,-0.072)(5.456,-0.096)(5.528,-0.144)(5.6,-0.192)(5.648,-0.264)
\psline[linewidth=1.6pt,linecolor=red](4.82,1.98)(4.9,1.98)(4.98,1.98)(5.06,1.98)(5.14,1.96)(5.22,1.96)(5.3,1.96)(5.38,1.96)(5.44,1.98)(5.48,2.04)(5.54,2.08)(5.62,2.1)(5.7,2.1)(5.78,2.1)(5.86,2.1)(5.94,2.1)(6.02,2.1)(6.08,2.14)(6.12,2.2)(6.22,2.22)(6.32,2.22)(6.4,2.22)(6.48,2.22)(6.56,2.22)(6.6,2.16)(6.64,2.1)(6.66,2.04)(6.66,1.96)(6.66,1.88)(6.66,1.8)(6.7,1.74)(6.76,1.7)(6.84,1.66)(6.92,1.66)(6.98,1.62)(7.04,1.56)(7.06,1.5)(7.04,1.42)(7.02,1.34)(6.98,1.28)(6.96,1.22)(6.9,1.18)(6.84,1.14)(6.8,1.08)(6.8,1.)(6.8,0.92)(6.8,0.82)(6.8,0.74)(6.8,0.66)(6.78,0.6)(6.76,0.54)(6.74,0.48)(6.7,0.4)(6.64,0.38)(6.56,0.38)(6.5,0.36)(6.42,0.36)(6.36,0.34)(6.28,0.34)(6.2,0.34)(6.12,0.34)(6.06,0.32)(6.02,0.26)(5.92,0.26)(5.84,0.26)(5.76,0.26)(5.7,0.24)(5.64,0.18)(5.56,0.18)(5.5,0.16)(5.42,0.16)(5.34,0.16)(5.24,0.16)(5.18,0.22)(5.18,0.3)(5.18,0.38)(5.16,0.44)(5.16,0.52)(5.14,0.6)(5.08,0.66)(5.02,0.72)(4.96,0.76)(4.9,0.82)(4.84,0.84)(4.8,0.9)(4.8,0.98)(4.8,1.06)(4.8,1.14)(4.84,1.2)(4.86,1.26)(4.88,1.32)(4.92,1.38)(4.92,1.46)(4.92,1.54)(4.92,1.64)(4.9,1.7)(4.84,1.72)(4.78,1.74)(4.78,1.82)(4.78,1.9)(4.84,1.94)(4.9,1.96)
\pscurve{->}(2.5,0.98)(3.38,1.26)(4.14,0.98)
\uput[u](3.38,1.26){$\phi$}
\pscurve{->}(2.5,0.6)(3.38,0.38)(4.14,0.6)
\uput[d](3.38,0.38){$\phi_k$}
\end{pspicture*}
\caption{The sets $\phi(U)$ (black) and $\phi_k(U)$ (red) for $k\geq k_0$.}
\end{figure}

\vspace{0.25cm}

We next establish that $\mathcal{C}_L$ is a dense subset of $C_+(I^2,\mathbb{R})$ for each $L>0$. 
Given $\varphi\in C_+(I^2,\mathbb{R})$ with image $[a,b]:=\varphi(I^2)\subset (0,\infty)$ and 
given $\varepsilon>0$ with $\varepsilon << b-a$, we want to construct a continuous 
function $\rho\in\mathcal{C}_{L}$ such that $||\rho-\varphi||_0<\varepsilon$. To do this, 
consider the interval $I_{\varepsilon} := \left(a,a+\varepsilon\right) \subset [a,b]$.  
%\begin{equation*}
%\begin{split}
%I_{\varepsilon}:=\left(a,a+\varepsilon\right)
%\end{split}
%\end{equation*}
Let $C_{\varepsilon}\subset I^2$ be a connected component of $\varphi^{-1}(I_{\varepsilon})$ and $S_{\varepsilon}$ a sufficiently small 
square contained in $C_{\varepsilon}$ . Let $\rho_{\varepsilon}:S_{\varepsilon}\longrightarrow I_{\varepsilon}$ be a continuous density that is not 
realizable as the Jacobian of an $L$-bi-Lipschitz homeomorphism (see Figure $4$). By the  Tietze extension theorem, there exists a continuous function  
$\hat{\rho}_{\varepsilon}:C_{\varepsilon}\longrightarrow [a,a+\varepsilon]$ such that $\hat{\rho}_{\varepsilon}|_{S_{\varepsilon}}=\rho_{\varepsilon}$ and 
\begin{center}
$\displaystyle\sup_{x\in C_{\varepsilon}}\hat{\rho}_{\varepsilon}(x)=\sup_{x\in S_{\varepsilon}}\rho_{\varepsilon}(x)$.
\end{center}

\psset{unit=1cm,algebraic=true,dimen=middle,dotstyle=o,dotsize=3pt 0,linewidth=0.8pt,arrowsize=3pt 2,arrowinset=0.25}
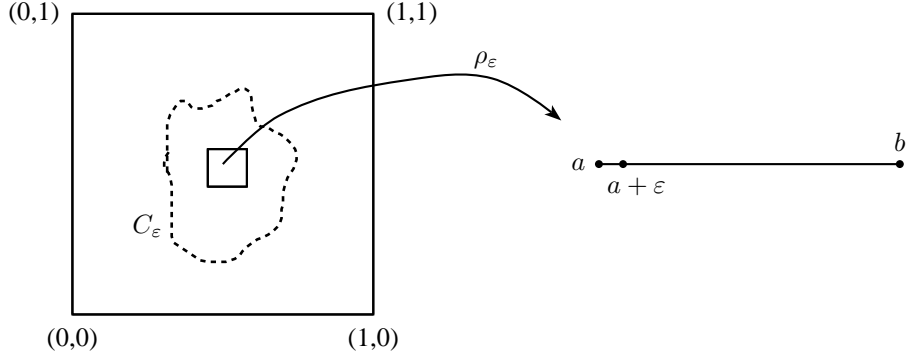
\begin{figure}[h]
\begin{pspicture}(-1.,-1.)(12.,5.)

\pspolygon(0.,0.)(4.,0.)(4.,4.)(0.,4.)
\pspolygon(1.8,2.2)(1.8,1.7)(2.32,1.7)(2.32,2.2)
\psline(0.,0.)(4.,0.)
\psline(4.,0.)(4.,4.)
\psline(4.,4.)(0.,4.)
\psline(0.,4.)(0.,0.)
\psline(7.,2.)(11.,2.)
\psline[linewidth=1pt, linestyle=dashed,dash=2pt 2pt](1.26,1.9)(1.26,2.02)(1.26,2.1)(1.26,2.2)(1.26,2.28)(1.26,2.38)(1.26,2.46)(1.26,2.54)(1.28,2.6)(1.3,2.66)(1.34,2.72)(1.38,2.78)(1.44,2.84)(1.52,2.84)(1.58,2.76)(1.64,2.74)(1.72,2.74)(1.78,2.76)(1.86,2.78)(1.92,2.84)(1.98,2.86)(2.04,2.9)(2.1,2.94)(2.16,2.96)(2.22,3.)(2.3,3.)(2.36,2.96)(2.38,2.88)(2.38,2.8)(2.4,2.7)(2.42,2.64)(2.44,2.56)(2.5,2.5)(2.58,2.5)(2.64,2.46)(2.7,2.42)(2.76,2.4)(2.82,2.36)(2.9,2.28)(2.96,2.2)(2.98,2.14)(2.98,2.06)(2.96,2.)(2.94,1.94)(2.9,1.88)(2.86,1.82)(2.84,1.76)(2.82,1.7)(2.8,1.62)(2.8,1.54)(2.8,1.46)(2.78,1.4)(2.78,1.28)(2.76,1.22)(2.72,1.16)(2.66,1.1)(2.6,1.06)(2.52,1.02)(2.46,1.)(2.4,0.96)(2.32,0.9)(2.26,0.84)(2.22,0.78)(2.16,0.74)(2.1,0.72)(2.04,0.7)(1.96,0.7)(1.88,0.7)(1.8,0.7)(1.7,0.72)(1.64,0.76)(1.58,0.8)(1.52,0.84)(1.46,0.88)(1.4,0.92)(1.34,0.96)(1.32,1.04)(1.32,1.12)(1.32,1.22)(1.32,1.3)(1.32,1.4)(1.32,1.48)(1.32,1.56)(1.32,1.64)(1.32,1.72)(1.3,1.82)(1.28,1.88)(1.22,1.94)(1.22,2.02)(1.24,2.1)(1.28,2.16)
\psline(1.8,2.2)(1.8,1.7)
\psline(1.8,1.7)(2.32,1.7)
\psline(2.32,1.7)(2.32,2.2)
\psline(2.32,2.2)(1.8,2.2)
\pscurve{->}(2,2)(2.74,2.64)(4.46,3.12)(5.58,3.14)(6.5,2.58)
\uput[d](0,0){(0,0)}
\uput[d](4,0){(1,0)}
\uput[l](0,4){(0,1)}
\uput[r](4,4){(1,1)}
\uput[l](7,2){$a$}
\uput[u](11,2){$b$}
\uput[d](7.5,2){$a+\varepsilon$}
\uput[u](5.5,3.1){$\rho_{\varepsilon}$}
\uput[d](1,1.5){$C_{\varepsilon}$}
\begin{scriptsize}
\psdots[dotstyle=*](7.,2.)
\psdots[dotstyle=*](11.,2.)
\psdots[dotstyle=*](7.32,2.)
\end{scriptsize}
\end{pspicture}
\caption{Construction of the function $\rho_{\varepsilon}:S_{\varepsilon}\longrightarrow [a,a+\varepsilon]$.}
\end{figure}

%The density of $\mathcal{C}_L$ over $C^+(I^2,\mathbb{R})$ is a consequence of the next lemma.\\

\begin{lemma}
%\noindent\textbf{Lemma 3 (Extension Lemma)}: 
%Given a positive number $\varepsilon<<b-a$ the continuous map $\hat{\rho_{\varepsilon}}:C_{\varepsilon}\longrightarrow[a,a+\varepsilon]$ defined as above. 
There exists a continuous function $\rho:I^2\longrightarrow\ [a,b]$ such that $\rho|_{I^2\setminus C_{\varepsilon}}=\varphi$, 
$\rho|_{S_{\varepsilon}}=\rho_{\varepsilon}$, and $||\rho-\varphi||_0<\varepsilon.$
\end{lemma}

\begin{proof}
Let $\{\varrho_1,\varrho_2\}$ be a partition of unity subordinate to the open cover $\{I^2\setminus S_{\varepsilon}, C_{\varepsilon}\}$ of $I^2$. 
Define $\rho:I^2\longrightarrow [a,b]$ by letting
\begin{center}
$\rho(x):=\varphi\varrho_1+\hat{\rho}_{\varepsilon} \varrho_2$.
\end{center}
By definition, it readily follows that $\rho|_{I^2\setminus C_{\varepsilon}} = \varphi$ and $\rho|_{S_{\varepsilon}}=\rho_{\varepsilon}$. 
In addition, since $\{\varrho_1,\varrho_2\}$ is a partition of unity, we have 
\begin{equation*}
\begin{split}
||\rho-\varphi||_0 &=||\varphi\varrho_1+\hat{\rho_{\varepsilon}}\varrho_2-\varphi||_0\\
                   &=||\varphi\varrho_1+\hat{\rho}_{\varepsilon}\varrho_2-\varphi(\varrho_1+\varrho_2)||_0\\
                   &=||(\hat{\rho_{\varepsilon}}-\varphi)\varrho_2||_0.
\end{split}
\end{equation*}
Now, since $\hat{\rho}_{\varepsilon} (U_{\varepsilon})\subset [a,a+\varepsilon]$, we have 
$\displaystyle\sup_{x\in C_{\varepsilon}\setminus S_{\varepsilon}}|(\hat{\rho}_{\varepsilon}(x)-\varphi(x))\varphi_2(x)|<\varepsilon$. 
Therefore, $||\rho-\varphi||_0<\varepsilon$, as desired.
\end{proof}

%%%%%%%%%%%%%%%%%%%%%%%%%%%%%%%%%%%

\section{The $L^{\infty}$ case}

We consider the set $L_+^{\infty}(I^2,\mathbb{R}):=\{\rho\in L^{\infty}(I^2,\mathbb{R})\mbox{ : } \rho(x) > 0 \mbox{ a.e.}\}$. 
Our goal is to prove the next
%Now, we will proof that the same conditions of Theorem $1$ holds in the $L^{\infty}$ context. This is given by the next Theorem:

\begin{theorem}
Let $\mathcal{C}_{\infty}$ be the set of all positive measurable functions $\rho\in L^{\infty}_+(I^2,\mathbb{R})$ such that there is no 
bi-Lipschitz map $f \!: I^2\longrightarrow\mathbb{R}^2$ satisfying $Jac(f) = \rho$ a.e. Then $\mathcal{C}_{\infty}$ is a thick subset 
of $L^{\infty}_+(I^2,\mathbb{R})$.
\end{theorem}

As in the continuous case, this theorem is a consequence of the next

\begin{proposition}Given $L>1$, let $\mathcal{C}_{L,{\infty}}$ be the set of all positive functions $\rho\in L_+^{\infty}(I^2,\mathbb{R})$ such that there is no 
$L$-bi-Lipschitz map $f:I^2\longrightarrow\mathbb{R}^2$ satisfying $Jac(f)=\rho$ a.e. Then $\mathcal{C}_{L,{\infty}}$ is an open dense subset of 
$L^{\infty}_+(I^2,\mathbb{R})$.
\end{proposition}

\begin{proof} By Remark 1, we know that $\mathcal{C}_{L,\infty}$ is open. Let us show that it is a dense subset of $L_+^{\infty}(I^2,\mathbb{R})$. 
Given $\varphi \in L_+^{\infty}(I^2,\mathbb{R})$, denote $b:= \sup \varphi (x)$. Fix $\varepsilon > 0$, and define $\varphi_{\varepsilon}$ as 
\begin{center}
$\varphi_{\varepsilon} (x) = \left\{
\begin{array}{c l}
\varphi (x) \quad \mbox{   if }  \varphi (x) \geq \varepsilon,\\
\\
\varepsilon \qquad \mbox{   if } \varphi (x) < \varepsilon.\\
\end{array}
\right.$
\end{center}
%By Lusin's theorem, for each $\mu > 0$, there is a continuous function $\hat{\varphi}_{\varepsilon}: I^2 \to [\varepsilon,b]$ that coincides 
%with $\varphi_{\varepsilon}$ over a compact set $K_{\mu}$ whose complement has measure at most $\mu$. Let $K:=K_{1/2}$, and 
%The idea is then to perturb $\varphi_{\varepsilon}$ on a very small neighborhood of a density point of such a compact set, say $K:=K_{1/2}$. 
Let $a \in [\varepsilon,b]$ be such that $\varphi_{\varepsilon}^{-1} ([a,a+\varepsilon])$ has positive measure. Let $y$ be a Lebesgue density 
point of $\varphi_{\varepsilon}^{-1} ([a,a+\varepsilon])$, which means that 
\begin{equation}
\displaystyle\lim_{\delta\to 0^+}\dfrac{\lambda(\varphi_{\varepsilon}^{-1} ([a,a+\varepsilon]) \cap S_{\delta}(y))}{\lambda(S_{\delta}(y))} = 1,
\end{equation}
where $S_{\delta}(y)$ denotes the square centered at $y$ and having sidelength $\delta$. An easy application of Lemma \ref{lema-epsilon} then 
shows the following: For a small-enough $\delta > 0$, there is a function $\rho_{\varepsilon}: S_{\delta} (y) \to [\varepsilon,b]$ that takes values in 
$[a,a+\varepsilon]$ for points in $\varphi_{\varepsilon}^{-1} ([a,a+\varepsilon]) \cap S_{\delta}(y)$ and coincides with $\varphi_{\varepsilon}$ on 
the complement of $\varphi_{\varepsilon}^{-1} ([a,a+\varepsilon]) \cap S_{\delta}(y)$ such that no $L$-bi-Lipchitz map 
$f: S_{\delta} (y) \to \mathbb{R}^2$ can have a Jacobian equal to $\rho_{\varepsilon}$. Let $\rho$ be defined by letting 
\begin{center}
$\rho (x) = \left\{
\begin{array}{c l}
\rho_{\varepsilon} (x) \quad \mbox{   if }  x \in S_{\delta}(y),\\
\\
\varphi_{\varepsilon} \qquad \mbox{   if } x \notin S_{\delta}(y).\\
\end{array}
\right.$
\end{center}
%being equal to $\rho_{\varepsilon}$ on $S_{\delta}(y)$ and to $\varphi_{\varepsilon}$ on the complement. 
By construction, $\rho$ belongs to $\mathcal{C}_{L,\infty}$, and $\| \rho - \varphi \|_{\infty} \leq 2 \varepsilon$. As this construction can 
be performed for any $\varepsilon > 0$, this shows that $\mathcal{C}_{L,\infty}$ is a dense subset of $L_+^{\infty}(I^2,\mathbb{R})$. 
\end{proof}
 
\vspace{0.3cm}

\noindent{\bf Acknowledgments.} 
This work was funded by the Fondecyt Project 1160541 and Anillo 1415 (PIA CONICYT) {\em Geometry at the Frontier}.

%%%%%%%%%%%%%%%%%%%%%%%%%%%%%%%%%%%%%%%%%%%%%%%%%%%%%%%%%%%%%%%%%%%%%%%%%%%%%%%%%

\vspace{0.4cm}

\noindent Rodolfo Viera (rodolfo.viera@usach.cl)\\
Dpto de Matem\'atica y Ciencia de la Computaci\'on\\
Universidad de Santiago de Chile\\
Alameda 3363, Estaci\'on Central, Santiago, Chile

\end{document}